\documentclass[12pt]{amsart}
\usepackage{float}
\usepackage{graphicx,verbatim,amssymb}
\usepackage[active]{srcltx}
\usepackage[cp1251]{inputenc}

\newtheorem{thm}{Theorem}[section]

\newtheorem{lem}[thm]{Lemma}
\newtheorem{prop}[thm]{Proposition}

\newtheorem*{mainthm}{Main Theorem}

\theoremstyle{definition}

\theoremstyle{remark}

\numberwithin{equation}{section}

\newcommand{\R}{\mathbb R}
\def\C{\mathbb{C}}

\def\Lc{\mathcal{L}}
\def\Oc{\mathcal{O}}

 \def\Zc{\mathcal{Z}}

\def\al{\alpha}
\def\si{\sigma}
\def\P{\mathbb{P}}

\def\0{\varnothing}
\def\sm{\setminus}
\def\ol{\overline}
\def\wt{\widetilde} 
\def\vp{\varphi}
\def\vk{\varkappa}

\renewcommand\ge{\geqslant}

\def\PGL{\mathrm{PGL}}

\begin{document}

\title[Planarizations]
{On maps taking lines to plane curves}
\author{Vsevolod Petrushchenko}
\author{Vladlen Timorin}

\address{Faculty of Mathematics and Laboratory of Algebraic Geometry,
National Research University Higher School of Economics,
7 Vavilova St 117312 Moscow, Russia}

\email{vtimorin@hse.ru}

\thanks{
The first named author is grateful to International Laboratory
of Decision Choice and Analysis at NRU HSE for partial financial support.
}
\thanks{
The second named author was partially supported by
the Dynasty Foundation grant, RFBR grants 13-01-12449, 13-01-00969,
and AG Laboratory NRU HSE, MESRF grant ag. 11.G34.31.0023.
Theorem \ref{t:nf} comprises the research funded by RScF grant 
14-21-00053.
}


\begin{abstract}
We study cubic rational maps that take lines to plane curves.
A complete description of such cubic rational maps concludes the classification
of all planarizations, i.e., maps taking lines to plane curves.
\end{abstract}
\maketitle

\section{Introduction}
Let $\P^n$ be the projective space of dimension $n$ over the field $\R$
of all real numbers or the field $\C$ of all complex numbers.
In \cite{Ti}, a \emph{planarization} was defined as a mapping
$\Phi:U\to\P^n$, where $U\subset\P^2$ is an open subset, such that
$\Phi(\lambda\cap U)$ is a subset of a hyperplane, for every line $\lambda\subset\P^2$.
Studying planarizations is closely related to studying maps taking lines
to curves of certain linear systems, cf. \cite{Ti}; a classical result of this type
is the M\"obius--von Staudt theorem \cite{Mob,vS} about maps taking lines to lines,
sometimes called the Fundamental Theorem of Projective Geometry.
We will always assume that the planarizations are sufficiently smooth, i.e.,
sufficiently many times differentiable.
If the ground field is $\C$, then we assume analyticity.
The main result of this paper is a complete description of all planarizations in case $n=3$.

\begin{mainthm}
  Let $\Phi:U\to\P^3$ be a planarization.
  Then there is a nonempty open subset $V\subset U$, for which
  the planarization $\Phi|_V:V\to\P^3$ is trivial, or co-trivial, or quadratic, or
  dual quadratic.
\end{mainthm}

We need to explain the terminology.
A planarization $\Phi:V\to\P^3$ is said to be \emph{trivial} if
$\Phi(V)$ is a subset of a plane.
A planarization $\Phi:V\to\P^3$ is said to be \emph{co-trivial} if
there exists a point $b\in\P^3$ such that, for every line $\lambda\subset\P^2$,
the set $\Phi(\lambda\cap V)$ lies in a plane containing $b$.
Of course, logically, co-trivial planarizations include trivial planarizations.
On the other hand, there are ``more'' trivial planarizations than co-trivial
planarizations that are not trivial.
This is one of the reasons for distinguishing trivial planarizations
as a separate class; the second reason being a partial duality between trivial and
co-trivial planarizations.
Trivial planarizations can be described in terms of an arbitrary map
from $\P^2$ to $\P^2$, and co-trivial planarizations can be described in
terms of an arbitrary function on $\P^2$.

A map $\Phi:V\to\P^3$ is said to be a \emph{quadratic rational map} if in some
(hence any) system of homogeneous coordinates it is given by quadratic
homogeneous polynomials.
In other words, there are quadratic homogeneous polynomials $Q_0$, $Q_1$,
$Q_2$ and $Q_3$ in $x_0$, $x_1$, $x_2$ such that $\Phi$ maps a point
with homogeneous coordinates $[x_0:x_1:x_2]$ to the point with homogeneous
coordinates $[y_0:y_1:y_2:y_3]$, where
$$
y_\alpha=Q_\alpha(x_0,x_1,x_2),\quad \alpha=0,1,2,3.
$$
It is easy to see that every quadratic rational map takes lines to conics,
in particular, every quadratic rational map is a planarization.
Some quadratic planarizations are neither trivial nor co-trivial.

Another class of examples is provided by duality.
Let $\Phi:V\to\P^3$ be a planarization.
Recall that the dual projective plane $\P^{2*}$ consists of lines in $\P^2$,
and the dual projective space $\P^{3*}$ consists of planes in $\P^3$.
Let $V^*$ be the subset of $\P^{2*}$ consisting of all lines $\lambda\subset\P^2$
with the following property: the set $\Phi(\lambda\cap V)$ lies in a unique plane $P_\lambda$.
Note that the set $V^*$ is open, possibly empty.
The dual planarization $\Phi^*:V^*\to\P^{3*}$ is by definition the map taking
$\lambda$ to $P_\lambda$.
Given a coordinate representation of $\Phi$, it is easy to
write explicit formulas for $\Phi^*$.
It turns out that a planarization dual to a quadratic rational map is a special kind of cubic rational map.
Such planarization is called \emph{dual quadratic}.
It is rather obvious that the duality is symmetric: if $\Phi:V\to\P^3$ is a
planarization, $\Phi^*:V^*\to\P^{3*}$ is the dual planarization with a
nonempty domain $V^*$, and $V^{**}\cap V\ne\0$, then $\Phi=\Phi^{**}$ on
$V^{**}\cap V$.

It is proved in \cite{Ti} that a planarization $\Phi:U\to\P^3$ that is neither
trivial nor co-trivial must be a rational map of degree two or three, at
least on some nonempty open subset of $U$.
Thus, to prove the Main Theorem, it suffices to describe all cubic
planarizations.
A cubic planarization is defined globally as a rational map from
$\P^2$ to $\P^3$ (it may have some points of indeterminacy).
Moreover, it suffices to assume that the ground field is $\C$.
Thus the description of all planarizations reduces to some question
of classical complex algebraic geometry.
In the following sections, we will answer this question.

It is natural to consider the following equivalence relation on the set of all
planarizations. Given two planarizations $\Phi:V\to\P^3$, $\Phi':V'\to\P^3$ we say that
they are \emph{equivalent} if they coincide on some nonempty open set after
suitable projective coordinate changes in $\P^2$ and in $\P^3$.
In other terms, there are projective automorphisms $\eta\in \PGL_2$, $\mu\in \PGL_3$,
and a nonempty open subset $W\subset V\cap\eta^{-1}(V')$
such that $\Phi=\mu\circ\Phi'\circ\eta$ on $W$.
In Section \ref{s:norf}, we describe all equivalence classes of
planarizations over real numbers by specifying a representative in each class.
These representatives will also be called \emph{normal forms} of planarizations,
so that every planarization is projectively equivalent to some normal form.
Of course, there are infinitely many classes of trivial and co-trivial planarizations.
These classes can be described by means of function parameters, i.e., they
depend on some arbitrary functions.
Other than that, there are 16 classes.
Our classification is based on the classification of equivalence classes of quadratic
rational maps \cite{CSS}.

\subsubsection*{Organization of the paper}
In Section \ref{s:cubmap}, we recall some basic properties of cubic rational
maps and the associated linear webs of plane cubic curves.
In Section \ref{s:plan}, we address specific properties of cubic planarizations.
The main result of this section is that a cubic planarization that
is neither trivial nor co-trivial and that has only finitely many points
of indeterminacy must map $\P^2$ many-to-one to its image surface.
In Section \ref{s:desc}, we complete the description of
all cubic planarizations thus proving the Main Theorem.
Section \ref{s:quad} is a digression needed for a classification of
all planarizations up to equivalence.
In this section, we classify all quadratic planarizations.
Finally, in Section \ref{s:norf}, we give a list of normal forms
for planarizations.

\section{Cubic maps and base points}
\label{s:cubmap}
In this section, all spaces and maps are defined over complex numbers.
Let $\Phi:\P^2\dashrightarrow\P^3$ be a cubic rational map sending
a point of $\P^2$ with homogeneous coordinates $[x_0:x_1:x_2]$
to a point of $\P^3$ with homogeneous coordinates
$[y_0:y_1:y_2:y_3]$, where
$$
y_\al=\vp_\al(x_0,x_1,x_2),\quad \al=0,1,2,3,
$$
and $\vp_\alpha$ is a homogeneous polynomial in three variables of degree 3.
Recall that an \emph{indeterminacy point} of $\Phi$ is a point $x$ in $\P^2$
such that $\vp_\alpha(x)=0$ for all $\al=0$, 1, 2, 3.
This is precisely a point that does not have an image under $\Phi$.

Recall that $\Phi$ defines a linear system $\Lc_{\Phi}$
of plane cubics of dimension 3
(a three-dimensional linear system is called a \emph{linear web}).
By definition, $\Lc_{\Phi}$ is generated by the cubics $\vp_\alpha=0$,
i.e., the equation of any cubic in $\Lc_{\Phi}$ has the form
$$
c_0\vp_0+c_1\vp_1+c_2\vp_2+c_3\vp_3=0,
$$
where the coefficients $c_0$, $c_1$, $c_2$, $c_3$ are complex numbers not
vanishing simultaneously, thus $[c_0:c_1:c_2:c_3]$ can be thought
of as a point in $\P^3$, or, in more invariant terms, as a point
in the dual projective space $\P^{3*}$ defining a plane $P$ in $\P^3$.
The plane cubic $\vk_P$ associated with $P$ and given by the equation displayed
above contains the set of all points $x\in\P^2$ such that $\Phi(x)\in P$.
Indeterminacy points of $\Phi$ are also called the \emph{base points}
of $\Lc_{\Phi}$.
We will write $B_\Phi$ for the set of all base points of $\Lc_\Phi$.
Every cubic from $\Lc_{\Phi}$ contains the set $B_\Phi$.

If $B_\Phi$ contains an irreducible curve $\beta$, then this curve has degree at most three.
Let $\xi=0$ be an irreducible equation defining the curve $\beta$.
Note that the equations of all cubics from $\Lc_\Phi$ are divisible by $\xi$.
Therefore, the restriction of $\Phi$ to the complement of $\beta$ coincides
with the rational map $\xi^{-1}\Phi$ of degree $3-\deg(\beta)$.
For this reason, we will mostly assume that $B_\Phi$ is a finite set of points.
There is a natural way of assigning multiplicity to every point $b\in B_\Phi$.
Namely, the multiplicity $m(b)$ is equal to the minimum intersection index
of two cubics in $\Lc_\Phi$ at $b$.
Since all cubics in $\Lc_\Phi$ pass through $b$, we have $m(b)\ge 1$.
We will write $|B_\Phi|$ for the number of points in $B_\Phi$ counting multiplicities.
In other terms, we have by definition
$$
|B_\Phi|=\sum_{b\in B_\Phi} m(b).
$$

In what follows, we will write $S_\Phi=\Phi(\P^2\sm B_\Phi)$ for the image of $\Phi$.
The following two propositions are classical and well known but we recall the
proofs for completeness.

Suppose that $B_\Phi$ is finite and that $S_\Phi$ has dimension two.
By \cite[Theorem II.7]{Bea} (elimination of indeterminacy),
there exists a compact projective surface $X$, a regular morphism $\pi:X\to\P^2$
that is a finite composition of blow-ups, and a regular morphism
$\Psi:X\to\P^3$ with the property $\Psi=\Phi\circ\pi$, which holds where
the right-hand side is defined.
By the Specialization Principle, cf. \cite[Theorem (3.25)]{Mum}, and since the
dimension of $\Psi(X)\supset S_\Phi$ is 2, a generic point in $\Psi(X)$ has
exactly $k$ preimages in $X$, where $k$ is the degree of the field
extension $\C(X)/\Psi^*\C(\Psi(X))$.
The difference $X\sm\pi^{-1}(\P^2\sm B_\Phi)$ consists of exceptional curves,
whose images under $\Psi$ lie in a proper Zariski closed subset of $\Psi(X)$.
Therefore, a generic point of $S_\Phi$ has exactly $k$ preimages in
$\P^2\sm B_\Phi$.
We will call the number $k$ the \emph{topological degree} of
$\Phi:\P^2\sm B_\Phi\to S_\Phi$.

\begin{prop}
\label{p:algdeg}
Suppose that $|B_\Phi|<\infty$ and $\dim(S_\Phi)=2$.
Let $k$ be the topological degree of $\Phi:\P^2\sm B_\Phi\to S_\Phi$.
Then the projective closure of $S_\Phi$ is a surface of degree 
$(9-|B_\Phi|)/k$, in particular, $|B_\Phi|<9$.
\end{prop}

\begin{proof}
Let $d$ denote the degree of the surface $\ol S_\Phi$, the closure of $S_\Phi$ in $\P^3$.
By the Kleiman transversality theorem, there is a proper Zariski closed set
$\Zc_1$ of lines in $\P^3$ such that every line $L\not\in\Zc_1$ intersects
the set $S_\Phi$ transversely at exactly $d$ points.
There is some at most one-dimensional exceptional subvariety $E$ of $S_\Phi$ such that
all points outside of $E$ have exactly $k$ preimages under
$\Phi:\P^2\sm B_\Phi\to S_\Phi$.
There is a proper Zariski closed set $\Zc_2$ of lines in $\P^3$ containing all
lines intersecting $E$.
Every line $L\subset\P^3$ defines a pencil $\Lc_\Phi(L)\subset\Lc_\Phi$
consisting of all cubics $\vk_P$, where $P$ runs through all planes containing $L$.
Let $\nu(L)$ be the sum of intersection indices of two generic curves
in $\Lc_\Phi(L)$ at points of $B_\Phi$.
Clearly, there is a proper Zariski closed set $\Zc_3$ of lines containing
all lines $L$ with $\nu(L)\ne |B_\Phi|$.

Consider a line $L$ not in $\Zc_1\cup\Zc_2\cup\Zc_3$.
Then the set $S_\Phi\cap L$ consists of $d$ transverse intersection points.
The line $L$ can be represented as the intersection of two planes $P_1$ and $P_2$.
Let $\vk_1$ and $\vk_2$ be the corresponding conics in $\Lc_\Phi$.
Since $L\not\in\Zc_1\cup\Zc_2$, the intersection $\vk_1\cap\vk_2$ is a
disjoint union of $dk$ transverse intersection points and the set $B_\Phi$.
Since $L\not\in\Zc_3$, the sum of intersection indices of $\vk_1$ and $\vk_2$
at points of $B_\Phi$ is equal to $|B_\Phi|$.
By the Bezout theorem, we have $9=3^2=dk+|B_\Phi|$.
\end{proof}

\begin{prop}
\label{p:curveB}
Let $\lambda\subset\P^2$ be a line such that $\Phi(\lambda\sm B_\Phi)$ lies
in a nonsingular conic. Then $\lambda$ intersects the set $B_\Phi$.
\end{prop}

\begin{proof}
  Let $C$ be the conic containing the set $\Phi(\lambda\sm B_\Phi)$.
  Similarly to the discussion presented above, there is a well-defined
  \emph{topological degree} $d_\lambda$ of the mapping $\Phi:\lambda\to C$. 
  A generic point of $C$ has exactly $d_\lambda$ preimages in $\lambda$,
  and these preimages have multiplicity one in the sense that 
  the differential of $\Phi:\lambda\to C$ does not vanish at these points.
  
  Let $\xi=0$ be an equation of a generic plane in $\P^3$.
  Then $\xi$ can be thought of as a section of the line bundle $\Oc_{\P^3}(1)$
  on $\P^3$.
  The restriction of $\xi$ to $C$ has two simple zeros.
  On the other hand, the section $\xi\circ\Phi$ of the line bundle $\Oc_\lambda(3)$
  on $\lambda$ has 3 zeros counting multiplicities.
  We obtain that $3=2 d_\lambda$, a contradiction.
  Alternatively, the proposition can be easily derived
  from Lemmas 2.6 and 2.7 of \cite{Ti}.
\end{proof}

\section{Planarizations}
\label{s:plan}
As before, we consider a cubic rational map $\Phi$.
We now assume that $\Phi$ is a \emph{planarization}, i.e., the $\Phi$-image of
every line in $\P^2$ is a subset of some plane in $\P^3$.
We say that $\Phi$ is \emph{strictly cubic} if there is no Zariski open subset $U\subset\P^2$
such that the restriction of $\Phi$ to $U$ coincides with some quadratic rational map.
As we have seen, for every strictly cubic planarization $\Phi$, the set $B_\Phi$ is
finite.

Consider any line $\lambda\subset\P^2$.
It is called \emph{non-special} if $\Phi(\lambda)$ is not a line.
By Proposition \ref{p:curveB}, or by the M\"obius--von Staudt theorem,
a generic line in $\P^2$ is non-special for $\Phi$.
Since $\Phi$ is a planarization, for every non-special line $\lambda$, there is a unique plane
$P_\lambda$ in $\P^3$ containing $\Phi(\lambda)$.
The preimage $\Phi^{-1}(P_\lambda)$ lies in a unique cubic curve
$\vk_\lambda\in\Lc_\Phi$ containing $\lambda$.
In fact, $\Phi^{-1}(P_\lambda)$ coincides with $\vk_\lambda\sm B_\Phi$.
Then $\vk_\lambda=\lambda\cup \si_\lambda$, where $\si_\lambda$ is a conic.
The curves $\Phi(\lambda)$ and $\Phi(\si_\lambda)$ are two plane curves in $\P^3$.
In this section, we will prove the following theorem.

\begin{thm}
 \label{t:deg}
Suppose that $\Phi:\P^2\dashrightarrow\P^3$ is a strictly cubic planarization, which is not
trivial and not co-trivial.
Then the topological degree of $\Phi:\P^2\sm B_\Phi\to S_\Phi$ is bigger than one.
\end{thm}

\subsection{General properties of cubic planarizations}
We assume in this section that $\Phi$ is a strictly cubic planarization such that the
dimension of $S_\Phi$ is two, and prove some general properties of $\Phi$.

\begin{lem}
\label{l:gen-cub}
  For a generic choice of $\lambda$, the curve $\Phi(\lambda)$ is cubic.
\end{lem}

\begin{proof}
  Suppose that the degree of $\Phi(\lambda)$ is at most 2, for a Zariski open
  set of lines $\lambda$.
  Then, by \cite[Lemma 2.8]{Ti}, the map $\Phi$ is not strictly cubic,
  i.e., all $\varphi_\al$ have a nontrivial common factor, a contradiction with
  the assumption that $B_\Phi$ is finite.
  The result also follows from Proposition \ref{p:curveB}.
\end{proof}

\begin{lem}
  \label{l:fibers}
  Suppose that $\Phi$ is not co-trivial.
  Then all fibers of $\Phi$ are finite, i.e., there is no 
  semi-algebraic subset of dimension one in $\P^2$ mapping to a point.
\end{lem}

\begin{proof}
 Suppose that $\Gamma\subset\P^2\sm B_\Phi$ is a semi-algebraic subset of
  dimension one such that $\Phi(\Gamma)$ is a point.
  Note that a generic line $\lambda\subset\P^2$ intersects $\Gamma$.
  Therefore, the plane $P_\lambda$ passes through the point $\Phi(\Gamma)$.
  It follows that $\Phi$ is co-trivial.
\end{proof}

Recall that, by our assumption, the image $S_\Phi$ has dimension two.
We will write $J_\Phi$ for the semi-algebraic set in $\P^2$,
on which the Jacobian of $\Phi:\P^2\sm B_\Phi\to S_\Phi$ vanishes.
In other terms, $J_\Phi$ consists of all points $p\in\P^2$ such that
the differential $d_p\Phi$ of $\Phi$ at $p$ is degenerate, i.e.,
has a nontrivial kernel.
Since $S_\Phi$ has dimension two, by the Sard lemma,
the Jacobian of $\Phi$ cannot vanish everywhere.
Therefore, the dimension of $J_\Phi$ is at most one.

\begin{prop}
  \label{p:Jac}
  Suppose that $\dim(J_\Phi)=1$.
  Then every component of $J_\Phi$ of dimension one is mapped to a subset of a plane.
\end{prop}

To prove Proposition \ref{p:Jac}, we need the following simple and general lemma:

\begin{lem}
\label{l:impos}
  If a germ of a holomorphic curve $T\subset\P^3$ has the property that
  all tangent lines of $T$ pass through some point $b\in\P^3$, then
  in fact $T$ lies in a line passing through $b$.
\end{lem}

\begin{proof}
  We may choose homogeneous coordinates in $\P^3$ so that $b=[0:0:0:1]$.
  In the affine chart $(x_1,x_2,x_3)\mapsto [1:x_1:x_2:x_3]$, the lines
  passing through $b$ are tangent to the vertical line field.
  The only integral curves of the vertical line field are vertical lines.
\end{proof}

\begin{proof}[Proof of Proposition \ref{p:Jac}]
Let $K$ be a component of $J_\Phi$, whose dimension is one.
The set $\Phi(K)$ is not a point, by Lemma \ref{l:fibers}.
Since $K$ is not mapped to a point under $\Phi$, the
restriction of $\Phi$ to $K$ has only finitely many critical points.
Note that, if $K$ lies in a line, then the statement follows from
the definition of a planarization.
Thus, we may assume that $K$ is not a subset of a line.

There are proper Zariski closed subsets $\Zc_1$, $\Zc_2$ and $\Zc_3$
of $K\times K$ with the following properties:
\begin{enumerate}
\item if $p$ or $q$ is a critical point of the restriction of $\Phi$
to $K$, then $(p,q)\in\Zc_1$;
\item if $(p,q)\not\in\Zc_2$, then the line connecting $p$ and $q$ is non-special;
\item let $\lambda$ be the line through $p$ and $q$; if the restriction
of the differential $d_p\Phi$ to the tangent line of $\lambda$ at $p$ vanishes
or the restriction of the differential $d_q\Phi$ to the tangent line of $\lambda$
at $q$ vanishes, then $(p,q)\in\Zc_3$.
\end{enumerate}

We now assume that $(p,q)$ does not belong to $\Zc_1\cup \Zc_2\cup \Zc_3$.
Since $p$, $q\in J_\Phi$, the curve $\Phi(K)$ is tangent to
$\Phi(\lambda)$ at points $\Phi(p)$ and $\Phi(q)$.
Moreover, since $(p,q)\not\in\Zc_1\cup\Zc_3$, the points $\Phi(p)$ and
$\Phi(q)$ are nonsingular for $\Phi(K)$ and $\Phi(\lambda)$, so that these
two varieties have well-defined tangent lines at $\Phi(p)$ and $\Phi(q)$.

It follows that the curve $\Phi(K)$ is tangent to the plane $P_\lambda$
at $\Phi(p)$ and $\Phi(q)$.
Tangent lines of $\Phi(K)$ at $\Phi(p)$ and $\Phi(q)$ lie in the
same plane $P_\lambda$, hence they intersect.
Since this is true for a Zariski dense set of pairs $(p,q)$, it follows
that every pair of tangent lines of $\Phi(K)$ intersect.
Fix two tangent lines $\Lambda_1$ and $\Lambda_2$ of $\Phi(K)$.
Any other tangent line $\Lambda$ of $\Phi(K)$ must intersect
both $\Lambda_1$ and $\Lambda_2$.
Thus, either $\Lambda$ lies in the plane containing $\Lambda_1\cup\Lambda_2$,
or $\Lambda$ passes through the intersection point $\Lambda_1\cap\Lambda_2$.
We see that the only possibilities for $\Phi(K)$ are that
\begin{enumerate}
\item all tangent lines
of this curve lie in the same plane, OR
\item all tangent lines of this curve
pass through the same point.
\end{enumerate}
In case (2), we have that $\Phi(K)$ is a subset of some line
by Lemma \ref{l:impos}, therefore, $\Phi(K)$ is a subset of a plane.
In case (1), we have that all tangent lines of $\Phi(K)$ belong
to the same plane, therefore, the curve $\Phi(K)$ itself lies in this
plane.
\end{proof}

\begin{prop}
  \label{p:Jac-fin}
  Suppose that $\Phi$ is strictly cubic, not trivial, not co-trivial, and not dual quadratic.
  Suppose also that the topological degree of $\Phi$ is equal to one.
  Then the set $J_\Phi$ is finite, possibly empty.
\end{prop}

To prove this proposition, we need the following lemma:

\begin{lem}
\label{l:Sdeg3}
  Suppose that $\Phi$ is strictly cubic, the topological degree of $\Phi$ is one,
  and $S_\Phi$ is a two-dimensional subset of some surface of degree 3.
  Then $\Phi$ is trivial or co-trivial.
\end{lem}

\begin{proof}
  By Proposition \ref{p:algdeg}, we have $|B_\Phi|=6$.
  For every line $\lambda\subset\P^2$ disjoint from $B_\Phi$, consider the
  corresponding conic $\si_\lambda$.
  Then $\si_\lambda$ contains the set $B_\Phi$.
  Moreover, we have $(\si_\lambda\cdot\varkappa)_{B_\Phi}\ge 6$, where $\varkappa$ is 
  any cubic from the linear web $\Lc_\Phi$, and $(\si_\lambda\cdot\varkappa)_{B_\Phi}$
  denotes the sum of the intersection multiplicities of $\si_\lambda$ and $\varkappa$
  at all points of $B_\Phi$.
  Indeed, if $\lambda\cap B_\Phi=\0$, then the inequality 
  $(\si_\lambda\cdot\varkappa)_{B_\Phi}\ge 6$ follows from
  $((\si_\lambda+\lambda)\cdot\varkappa)_{B_\Phi}\ge 6$.
  The general case follows from the upper-semicontinuity of the intersection multiplicities.

  If a line $\lambda'$ is disjoint from $B_\Phi$, then 
  $(\si_\lambda\cdot\si_{\lambda'})_{B_\Phi}\ge 6$.  
  On the other hand, two different conics either share a line component
  or intersect by at most 4 points, counting multiplicities.
  It follows that all conics $\si_\lambda$ share a line component $\lambda_0$,
  in particular, all $\si_\lambda$ have a point $a\notin B_\Phi$ in common.
  Then, for every line $\lambda\subset\P^2$, the plane $P_\lambda$ containing the 
  image of $\lambda$ contains also $\Phi(a)$,
  which means that $\Phi$ must be co-trivial.
\end{proof}

\begin{proof}[Proof of Proposition \ref{p:Jac-fin}]
  Assume the contrary: there is a component $K$ of $J_\Phi$ that has dimension one.
  By Proposition \ref{p:Jac}, the image $\Phi(K)$ is a plane curve.
  Suppose that neither $K$ not $\Phi(K)$ is a subset of a line.
  We will write $P$ for the plane containing $\Phi(K)$.
  By the proof of Proposition \ref{p:Jac}, the image of a generic line
  $\lambda\subset\P^2$ under the map $\Phi$ is tangent to
  $\Phi(K)$ at two or more points.
  Moreover, we may assume that the tangent lines of $\Phi(K)$ at
  these points do not coincide and therefore define a unique plane.
  It follows that $P_\lambda=P$.
  Since $\lambda$ is generic, this implies that $\Phi$ is trivial.

  Suppose now that $\Phi(K)$ lies in a line $L\subset\P^3$.
  Then, since, for a generic line $\lambda\subset\P^2$, the curve
  $\Phi(\lambda)$ is tangent to $\Phi(K)$, the plane $P_\lambda$
  must contain $L$.
  It follows that $\Phi$ is co-trivial.

  Finally, suppose that $K$ is a subset of a line but $\Phi(K)$
  is not a subset of a line.
  Then $\Phi(K)$ lies in a plane algebraic curve $\Xi$ of degree two or three.
  Consider the dual planarization $\Phi^*$.
  If $\Phi$ is neither trivial nor co-trivial, then $\Phi^*$ is defined on
  some nonempty Zariski open subset of $\P^{2*}$.
  If $\Phi^*$ is trivial, then $\Phi$ must be co-trivial.
  If $\Phi^*$ is co-trivial, then $\Phi$ must be trivial.
  Thus we may assume that $\Phi^*$ is neither trivial nor co-trivial.
  Note that, for a generic line $\lambda\subset\P^2$, the image
  $\Phi(\lambda)$ is tangent to $\Phi(K)$ at the point $\Phi(\lambda\cap K)$.
  It follows that the plane $P_\lambda$ is tangent to $\Phi(K)$.
  We see that the image $\Phi^*(\P^{2*}\sm B_{\Phi^*})$ lies in the
  set of all tangent planes of $\Xi$.

  The set $\Xi^*$ of all tangent planes of $\Xi$ is a cone in $\P^{3*}$
  (indeed, every plane in $\Xi^*$ is an element of a linear pencil of
  planes, i.e., of a line in $\P^{3*}$, containing the plane of $\Xi$).
  A plane section of the cone $\Xi^*$ not passing through the vertex
  of this cone is a curve projectively equivalent to the dual curve of $\Xi$.
  Since $\Phi^*$ is a planarization, it must be a cubic rational map by \cite{Ti}.
  It follows that the projectively dual curve of $\Xi$ has degree at most three.
  Then the degree of the surface $\Xi^*$ is also at most three.
  If the image of $\P^{2*}\sm B_{\Phi^*}$ under $\Phi^*$ has dimension one,
  then this image lies in a plane, hence $\Phi^*$ is trivial.
  Thus we may assume that $\Phi^*(\P^{2*}\sm B_{\Phi^*})$ has dimension two,
  i.e., includes an open subset of $\Xi^*$.
  It now follows from Lemma \ref{l:Sdeg3} that the planarization $\Phi$
  is trivial or co-trivial.
\end{proof}

\subsection{Proof of Theorem \ref{t:deg}}
In this section, we assume that $\Phi$ satisfies all assumptions of Theorem \ref{t:deg},
namely, that $\Phi$ is not trivial, not co-trivial, and is strictly cubic.
It follows that $S_\Phi$ has dimension two and that the set $B_\Phi$ is finite.

First note that, if $\Phi$ is dual quadratic, then the conclusion of
Theorem \ref{t:deg} holds.
Indeed, let $\Phi^*$ be the dual planarization of $\Phi$.
It is defined on some nonempty Zariski open subset of $\P^{2*}$.
Since $\Phi$ is dual quadratic, the map $\Phi^*$ is a quadratic rational map.
Recall that lines in $\P^{2*}$ correspond to points in $\P^2$:
namely, a point $a\in\P^2$ defines the line $a^*\in\P^{2*}$
consisting of all lines in $\P^2$ passing through $a$.
If a line $a^*\subset\P^{2*}$ is non-special for $\Phi^*$, we will write
$P^*_a$ for the plane in $\P^3$ containing the set $\Phi^*(a^*\sm B_{\Phi^*})$.
The plane $P^*_a$, of course, identifies with the point $\Phi(a)$.
The plane $P^*_a$ in $\P^3$ defines a conic in $\P^2$ containing $a^*$.
This conic consists of $a^*$ and another line $a_1^*$.
Clearly, we have $a_1\ne a$ for a generic $a\in\P^2$ (there is no nontrivial
\emph{linear} systems of double lines in $\P^2$) and that $\Phi(a_1)=\Phi(a)$.
It follows that the topological degree of $\Phi$ is bigger than one.
We may now assume that $\Phi$ is not dual quadratic.

\begin{prop}
  \label{p:cusp}
  It is impossible that, for a Zariski dense set of lines $\lambda\subset\P^2$,
  the images $\Phi(\lambda)$ are cuspidal cubics.
\end{prop}

\begin{proof}
We will write $\gamma$ for the Zariski closure of the set of all
points $p$ with the following property: there is a line $\lambda\ni p$, whose
$\Phi$-image is a cuspidal cubic curve, the point $\Phi(p)$ being the cusp of this curve.
If $p$ and $\lambda$ are as above, then $p\in J_\Phi$.
It follows that $\gamma\subset J_\Phi$ has dimension at most one.
Since a generic line intersects $\gamma$, we conclude that $\gamma$ is a curve
rather than a finite set of points.
Assuming that $\Phi$ is not trivial, not co-trivial,
not dual quadratic, and is of topological degree one
we get a contradiction with Proposition \ref{p:Jac-fin}, which states that
$J_\Phi$ and hence $\gamma$ are finite sets.
\end{proof}

By Lemma \ref{l:gen-cub}, for a generic line $\lambda\subset\P^2$, the set 
$\Phi(\lambda)$ is a cubic curve.
A plane rational cubic curve is either nodal or cuspidal.
Thus we have two cases.
Suppose first that, for a generic line $\lambda$, the cubic $\Phi(\lambda)$ is cuspidal.
Then, by Proposition \ref{p:cusp}, the planarization $\Phi$ is trivial, or co-trivial,
or dual quadratic.
Thus we may assume that, for a non-empty Zariski open set of lines $\lambda$, the
cubic $\Phi(\lambda)$ is nodal.
For every line $\lambda\subset\P^2$ such that $\Phi(\lambda)$ is a nodal cubic, we
let $\Sigma(\lambda)$ be the set of points of $\lambda$ mapping to the singular point of $\Phi(\lambda)$.
Thus the set $\Sigma(\lambda)$ consists of two points, and these
two points are mapped to the node of the nodal cubic $\Phi(\lambda)$.
Let $\Gamma$ be the Zariski closure of the union of $\Sigma(\lambda)$ over all lines $\lambda\subset\P^2$ such that $\Phi(\lambda)$ is a nodal cubic.

\begin{lem}
  \label{l:GammaP2}
  The set $\Gamma$ coincides with the whole of $\P^2$.
\end{lem}

\begin{proof}
Assume the contrary: the set $\Gamma$ has dimension one or less.
Then the set $\Zc$ of lines $\lambda$ such that $\Phi:\lambda\cap\Gamma\to\P^3$
is not injective has dimension at most one.
Indeed, given a point $a\in\Gamma$, there are only finitely many lines
connecting $a$ with some other point in the finite set $\Phi^{-1}(\Phi(a))$
(the latter set is finite by Lemma \ref{l:fibers}).
Consider a line $\lambda\not\in\Zc$.
Moreover, we may assume that $\Phi(\lambda)$ is a nodal cubic.
Then the set $\Sigma(\lambda)$ contains some point $b\ne a$.
This is a contradiction with the fact that $\Phi$ is injective on the
set $\lambda\cap\Gamma$.
\end{proof}

It follows from Lemma \ref{l:GammaP2} that the topological degree of the map $\Phi$
is strictly bigger than one, thus Theorem \ref{t:deg} is proved.

\section{Description of cubic planarizations}
\label{s:desc}
In this section, we give a complete description of cubic planarizations thus
completing the description of all planarizations.
We will assume throughout this section that $\Phi$ is a strictly cubic planarization
that is neither trivial nor co-trivial.
Then $S_\Phi$ has dimension two, and the set $B_\Phi$ is finite.
By Theorem \ref{t:deg}, the topological degree of the map
$\Phi:\P^2\sm B_\Phi\to S_\Phi$ is at least two.
We can now make this result stronger.

\begin{prop}
   If $\Phi$ is not dual quadratic, then
   the topological degree of the map $\Phi:\P^2\sm B_\Phi\to S_\Phi$ is
   equal to three.
\end{prop}

\begin{proof}
Consider the dual planarization $\Phi^*$.
Recall that it is defined on some nonempty Zariski open subset of $\P^{2*}$.
By the classification of planarizations, the map $\Phi^*$ must be cubic.
Moreover, $\Phi^*$ is neither trivial, nor co-trivial
(otherwise $\Phi$ would be trivial or co-trivial).

Consider the set $B_{\Phi^*}\subset\P^{2*}$ of all indeterminacy points of $\Phi^*$.
Since $\Phi$ is not dual quadratic, the planarization $\Phi^*$ is strictly cubic.
It follows that the set $B_{\Phi^*}$ is finite.

All facts established earlier for $\Phi$ apply also to $\Phi^*$.
In particular, a generic fiber of the map $\Phi^*$ consists of at least
two points, and a generic line $a^*\subset\P^{2*}$ is mapped to
a nodal cubic under $\Phi^*$.
If a line $a^*\subset\P^{2*}$ is non-special for $\Phi^*$, then we will write
$P^*_{a^*}$ for the unique plane in $\P^{3*}$ containing $\Phi^*(a^*)$.
Recall that $P^*_{a^*}$ is identified with $\Phi(a)$ under the 
natural identification between $\P^{3**}$ and $\P^3$.
Similarly to the properties of $\Phi$, the full preimage of $P^*_{a^*}$ under $\Phi^*$
is a cubic curve $\vk^*_{a^*}$ consisting of $a^*$ and some conic
$\sigma^*_{a^*}$.

Since the topological degree of $\Phi$ is at least two, we know that, for a generic line
$a_1^*\subset\P^{2*}$, there is another line in $\P^{2*}$ mapping
to the same plane $P^*_{a_1^*}$ under $\Phi^*$.
Hence the conic $\si^*_{a_1^*}$ splits into the union of two lines.
We will write $a_2^*$ and $a_3^*$ for these two lines.
Thus, the cubic $\vk^*_{a_1^*}$ splits into the union of the three lines
$a_1^*$, $a_2^*$ and $a_3^*$.
Generically, these three lines are different, and they map to
the same plane $P^*_{a_1^*}$.
This property of $\Phi^*$ translates to the following property of $\Phi$:
a generic point of $S_\Phi$ (corresponding to the plane $P^*_{a_1^*}$)
has exactly three preimages $a_1$, $a_2$, $a_3$.
\end{proof}

We now assume that the topological degree of the map $\Phi:\P^2\sm B_\Phi\to S_\Phi$
is equal to three, and the same is true for the dual planarization $\Phi^*$.
By Proposition \ref{p:Jac-fin}, we may also assume that $J_\Phi$ is finite.
Let $d$ denote the degree of the surface $\ol S_\Phi$.
By Proposition \ref{p:algdeg}, we have $3d=9-|B_\Phi|$.
It follows that $d$ is at most three, i.e., the surface $\ol S_\Phi$ is at most cubic.

Suppose that $\ol S_\Phi$ is a cubic surface.
It follows that the set $B_\Phi$ is empty, hence $S_\Phi$ is compact.
It is a classical fact that $S_\Phi$ contains at least one line
(recall that a smooth cubic surface contains 27 lines, and any cubic surface can
be approximated by smooth cubic surfaces).
Let $L$ be a line contained in $S_\Phi$.
Since $J_\Phi$ is finite, the $\Phi$-preimage of a generic point in $L$
consists of exactly three points.
This means that there are three different lines $\lambda_0$, $\lambda_1$ and 
$\lambda_2$ mapping to $L$. 
Indeed, the set $\Phi^{-1}(L)$ is contained in a cubic curve $\Phi^{-1}(P)$,
where $P\subset\P^3$ is any plane containing the line $L$.
On the other hand, there are three distinct elements of $\P^{2*}$ mapping to 
$P\in\P^{3*}$ under $\Phi^*$, for a generic choice of $P\supset L$. 
We see that $\Phi^{-1}(L)=\Phi^{-1}(P)$ is a union of three distinct lines,
each mapping one-to-one to $L$.
This leads to a contradiction, because we can take $P$ passing through $L$
and some other (generic) point of $S_\Phi$; then $\Phi^{-1}(L)\ne\Phi^{-1}(P)$.
The thus obtained contradiction with the assumption that
the topological degree of $\Phi$ is three concludes the proof of the Main Theorem.

\section{Quadratic planarizations}
\label{s:quad}

Throughout this section, we suppose that
$\Phi:\P^2\dashrightarrow\P^3$ is a quadratic rational map
such that the image of $\Phi$ lies in some quadratic surface $S$
but does not lie in a plane.
We will classify all such quadratic maps up to projective equivalence.
The classification must be classical but we failed to find a modern reference.
The following theorem describes the classification over $\C$.

\begin{thm}
\label{t:qclass}
Suppose that the ground field is $\C$.
Then $\Phi$ is equivalent to one and only one of the following three maps:
\begin{align*}
\Phi_{1}:&[x_0:x_1:x_2]\mapsto [x_0^2:x_0x_1:x_0x_2:x_1x_2],\\
\Phi_{2}:&[x_0:x_1:x_2]\mapsto [x_0^2:x_0x_1:x_1^2:x_0x_2],\\
\Phi_{3}:&[x_0:x_1:x_2]\mapsto [x_0^2:x_0x_1:x_1^2:x_2^2].
\end{align*}
The planarizations $\Phi_1$ and $\Phi_2$ are co-trivial.
The dual planarization of $\Phi_3$ is equivalent to $\Phi_3$.
\end{thm}

The corresponding real classification differs only in that the
complex equivalence class of $\Phi_1$ splits into two real equivalence
classes $\Phi_{1a}$ and $\Phi_{1b}$.

\begin{thm}
\label{t:qclassR}
Suppose that the ground field is $\R$.
Then $\Phi$ is equivalent to one and only one of the following four maps:
\begin{align*}
\Phi_{1a}:&[x_0:x_1:x_2]\mapsto [x_0^2:x_0x_1:x_0x_2:x_1x_2],\\
\Phi_{1b}:&[x_0:x_1:x_2]\mapsto [x_0^2:x_0x_1:x_0x_2:x_1^2+x_2^2],\\
\Phi_{2}:&[x_0:x_1:x_2]\mapsto [x_0^2:x_0x_1:x_1^2:x_0x_2],\\
\Phi_{3}:&[x_0:x_1:x_2]\mapsto [x_0^2:x_0x_1:x_1^2:x_2^2].
\end{align*}
The planarizations $\Phi_{1a}$, $\Phi_{1b}$ and $\Phi_2$ are co-trivial.
The dual planarization of $\Phi_3$ is equivalent to $\Phi_3$.
\end{thm}

In Section \ref{ss:C}, we prove Theorem \ref{t:qclass}, and
in Section \ref{ss:R}, we prove Theorem \ref{t:qclassR}.

\subsection{Complex classification}
\label{ss:C}
In this section, we assume that the ground field is $\C$.
The proof of Theorem \ref{t:qclass} consists of several lemmas.
We first assume that the quadric $S$ is non-degenerate.

\begin{lem}
\label{l:nondeg-par}
If the surface $S$ is given by the equation $u_0u_1=u_2u_3$ with respect
to some system of homogeneous coordinates $[u_0:u_1:u_2:u_3]$ in $\P^3$,
then $\Phi$ has the form
$$
[x_0,x_1,x_2]\mapsto [\psi_0\psi_1:\psi_2\psi_3:\psi_0\psi_2:\psi_1\psi_3],
$$
where $\psi_\alpha$, $\alpha=0$, $\dots$, $3$,
are homogeneous linear forms in $x_0$, $x_1$, $x_2$.
\end{lem}

\begin{proof}
The map $\Phi$ can be written in coordinates as
$u_\alpha=\varphi_\alpha(x_0,x_1,x_2)$, where $x_0$, $x_1$, $x_2$
are homogeneous coordinates in $\P^2$, and the index $\alpha$ runs from 0 to 3.

We claim that every quadratic polynomial $\varphi_\alpha$ is reducible.
Indeed, if one of these polynomials, say, $\varphi_0$ is irreducible,
then, by the unique factorization property, $\varphi_2$ or $\varphi_3$
is divisible by $\varphi_0$, hence is proportional to $\varphi_0$.
It follows however that the image of $\Phi$ lies in a plane, a contradiction.
Thus every $\varphi_\alpha$ is a product of two linear factors.
We write $\varphi_0$ as $\psi_0\psi_1$, where $\psi_0$ and $\psi_1$
are linear homogeneous polynomials in $x_0$, $x_1$, $x_2$.
Then $\varphi_2$ or $\varphi_3$ is divisible by $\psi_0$.
Relabeling $\varphi_2$ and $\varphi_3$ if necessary, we may assume that
$\varphi_2$ is divisible by $\psi_0$.
Set $\varphi_2=\psi_0\psi_2$, where $\psi_2$ is some linear polynomial.
It now follows from the identity $\varphi_0\varphi_1=\varphi_2\varphi_3$ that
$\psi_1\varphi_1=\psi_2\varphi_3$.
We see that $\varphi_3$ is divisible by $\psi_1$, therefore, $\varphi_3$
can be written as $\psi_1\psi_3$.
It follows that $\varphi_1=\psi_2\psi_3$.
\end{proof}

We can now classify all maps $\Phi$, for which $S$ is non-singular.

\begin{lem}
\label{l:C-nonsing}
Suppose that $S$ is nonsingular.
Then $\Phi$ is equivalent to the following map:
\begin{align*}
\Phi_{1}:&[x_0:x_1:x_2]\mapsto [x_0^2:x_0x_1:x_0x_2:x_1x_2].
\end{align*}
In particular, $\Phi$ is co-trivial.
\end{lem}

\begin{proof}
There is a system of homogeneous coordinates $u_0$, $u_1$, $u_2$, $u_3$
in the target space $\P^3$ such that the surface $S$ is given by the
equation $u_0u_1=u_2u_3$.
By Lemma \ref{l:nondeg-par}, the map $\Phi$ has the form
$[x_0,x_1,x_2]\mapsto [\psi_0\psi_1:\psi_2\psi_3:\psi_0\psi_2:\psi_1\psi_3]$,
where $\psi_\alpha$ are linear forms in $x_0$, $x_1$, $x_2$.
The set of indeterminacy points of $\Phi$ is equal to
$B_\Phi=\{\psi_0=\psi_3=0\}\cup\{\psi_1=\psi_2=0\}$.
Indeed, if $\psi_0\ne 0$, then we must have $\psi_1=\psi_2=0$,
and if $\psi_1\ne 0$, then we must have $\psi_0=\psi_3=0$.

We claim that the set $B_\Phi$ consists of exactly two points.
The system of equations $\psi_0=\psi_3=0$ defines a point $a$.
Indeed, otherwise the linear functionals $\psi_0$ and $\psi_3$
must be proportional, and we may assume $\psi_0=\psi_3$.
In this case, we have $\varphi_0=\varphi_3$, which means that
the image of $\Phi$ lies in a plane section of $S$, a contradiction with our assumption.
Similarly, the system of equations $\psi_1=\psi_2=0$ defines a point $b$.
It remains to show that $a\ne b$.
Indeed, otherwise the map $\Phi$ factors through the central projection
of $\P^2\sm\{a\}$ onto $\P^1$.
It follows that the image of $\Phi$ lies in a conic, a contradiction with
our assumption.
Thus we have $B_\Phi=\{a,b\}$.

Consider the linear web of conics $\Lc_\Phi$ associated with $\Phi$.
All conics of $\Lc_\Phi$ pass through $a$ and $b$.
On the other hand, the linear system $\Lc$ of all conics passing through $a$
and $b$ has dimension 3.
Therefore, $\Lc_\Phi=\Lc$.
We can now choose homogeneous coordinates $[x_0:x_1:x_2]$ in $\P^2$
so that $a=[0:1:0]$ and $b=[0:0:1]$.
Then $\Lc$ is spanned by the following degenerate conics: $x_0^2=0$,
$x_0x_1=0$, $x_0x_2=0$ and $x_1x_2=0$.
The map $\Phi$ corresponding to this choice of generators coincides with $\Phi_1$,
as desired.
The planarization $\Phi_1$ is co-trivial: indeed, every line is mapped
under $\Phi_1$ to a plane passing through $[0:0:0:1]$.
\end{proof}

Continuing the complex classification of quadratic planarizations, we now assume
that $S$ is contained in a degenerate quadric.

\begin{lem}
\label{l:cone-par}
  Suppose that $S$ is given by the
  equation $u_1^2=u_0u_2$ with respect to some system of homogeneous coordinates
  $[u_0:u_1:u_2:u_3]$ in $\P^3$.
  Then the map $\Phi$ has the form
  $$
  [x_0:x_1:x_2]\mapsto [x_0^2:x_0x_1:x_1^2,\varphi_3(x_0,x_1,x_2)],
  $$
  where $\varphi_3$ is some homogeneous quadratic form in the variables
  $x_0$, $x_1$, $x_2$.
\end{lem}

\begin{proof}
Suppose that $\Phi$ is given by the equations $u_\alpha=\varphi_\alpha(x_0,x_1,x_2)$.
As before, we argue that $\varphi_1$ is reducible, otherwise it would be
proportional either to $\varphi_0$ or to $\varphi_2$.
Similarly, $\varphi_0$ and $\varphi_2$ are reducible.
We can write $\varphi_1$ as $\psi_0\psi_1$, where $\psi_0$ and $\psi_1$ are linear functions.
Then $\varphi_0$ or $\varphi_2$ is divisible by $\psi_0$; we may assume the former
and write $\varphi_0=\psi_0\tilde\psi_0$.
It follows from the equation $\varphi_1^2=\varphi_0\varphi_2$ that
$\psi_0\psi_1^2=\tilde\psi_0\varphi_2$.
Therefore, $\tilde\psi_0$ is proportional to $\psi_0$ or to $\psi_1$.
In the former case, we have $\varphi_0=\psi_0^2$ and $\varphi_2=\psi_1^2$,
up to a projective coordinate change in $\P^3$ (multiplying the homogeneous
coordinates by different constants).
In the latter case, $\varphi_0$ would be proportional to $\varphi_1$,
a contradiction with our assumption.

Thus we have
$\varphi_1=\psi_0\psi_1$, $\varphi_0=\psi_0^2$, $\varphi_2=\psi_1^2$
for some non-proportional linear forms $\psi_0$, $\psi_1$.
We can choose the homogeneous coordinates $[x_0:x_1:x_2]$ in $\P^2$
so that $\psi_0=x_0$ and $\psi_1=x_1$.
The map $\Phi$ now takes the form
$[x_0:x_1:x_3]\mapsto [x_0^2:x_0x_1:x_1^2:\varphi_3(x_0,x_1,x_2)]$,
where $\varphi_3$ is a quadratic form in the variables $x_0$, $x_1$, $x_2$,
as desired.
\end{proof}

The following lemma provides normal forms for $\Phi$ in the case, where
$S$ is an irreducible cone.

\begin{lem}
\label{l:C-sing}
  Suppose that $S$ is a singular
  irreducible quadric, i.e., a quadratic cone.
  Then $\Phi$ is equivalent to at least one of the maps
\begin{align*}
  \Phi_{2}:&[x_0:x_1:x_2]\mapsto [x_0^2:x_0x_1:x_1^2:x_2x_0]\\
  \Phi_{3}:&[x_0:x_1:x_2]\mapsto [x_0^2:x_0x_1:x_1^2:x_2^2].
\end{align*}
\end{lem}

\begin{proof}
There is a homogeneous coordinate system $u_0$, $u_1$, $u_2$, $u_3$
in the space $\P^3$ such that the cone $S$ is given by the equation $u_1^2=u_0u_2$.
By Lemma \ref{l:cone-par}, we may assume that the map $\Phi$ has the form
$[x_0:x_1:x_3]\mapsto [x_0^2:x_0x_1:x_1^2:\varphi_3(x_0,x_1,x_2)]$,
where $\varphi_3$ is a quadratic form in the variables $x_0$, $x_1$, $x_2$.
We may change $\varphi_3$ by adding any linear combination of $x_0^2$, $x_0x_1$, $x_1^2$,
i.e., by adding any quadratic form in $x_0$, $x_1$ only:
this can be implemented by means of a projective coordinate change in the target
space $\P^3$.
Thus we may assume that
$\varphi_3(x_0,x_1,x_2)=x_2(a_0x_0+a_1x_1+a_2x_2)$.

Suppose first that $a_2=0$.
Then at least one of the coefficients $a_0$, $a_1$ is nonzero.
Assume e.g. that $a_0\ne 0$ (the case $a_1\ne 0$ is obtained
from this case by interchanging $x_0$ and $x_1$).
Then we set $\wt x_0=a_0x_0+a_1x_1$, $\wt x_1=x_1$, $\wt x_2=x_2$.
In the new variables, the map $\Phi$ has the form
$\Phi:[\wt x_0:\wt x_1:\wt x_2]\mapsto [U_0:U_1:U_2:\wt x_2\wt x_0]$,
where $U_0$, $U_1$ and $U_2$ are linearly independent quadratic
forms in $\wt x_0$, $\wt x_1$.
Since the space of quadratic forms in $\wt x_0$, $\wt x_1$ is three-dimensional,
the monomials $\wt x_0^2$, $\wt x_0\wt x_1$, $\wt x_1^2$ can be
represented as linear combinations of $U_0$, $U_1$, $U_2$.
Therefore, changing homogeneous coordinates in the target space $\P^3$,
we can reduce $\Phi$ to the form
$\Phi:[\wt x_0:\wt x_1:\wt x_2]\mapsto
[\wt x_0^2:\wt x_0\wt x_1:\wt x_1^2:\wt x_2\wt x_0]$,
i.e., to the form $\Phi_2$.

Suppose now that $a_2\ne 0$.
Then we make the following change of variables:
$x_0=\wt x_0$, $x_1=\wt x_1$,
$x_2=c_0\wt x_0+c_1\wt x_1+c_2\wt x_2$.
In the new variables, the map $\Phi$ has the form
\begin{align*}
[\wt x_0:\wt x_1:\wt x_2]\mapsto [\wt x_0^2:\wt x_0\wt x_1:\wt x_1^2:
\wt\varphi_3(\wt x_0,\wt x_1,\wt x_2)],\\
\wt\varphi_3(\wt x_0,\wt x_1,\wt x_2)=
(c_0\wt x_0+c_1\wt x_1+c_2\wt x_2)\left((a_0+a_2c_0)\wt x_0+(a_1+a_2c_1)\wt x_1+
a_2c_2\wt x_2\right)=\\
=c_2\wt x_2\left((a_0+2a_2c_0)\wt x_0+(a_1+2a_2c_1)\wt x_1+a_2c_2\wt x_2\right)+\dots.
\end{align*}
The dots mean a quadratic form in $\wt x_0$, $\wt x_1$.
We now set
$$
c_0=-\frac{a_0}{2a_2},\quad c_1=-\frac{a_1}{2a_2},\quad c_2=\frac 1{\sqrt{a_2}}
$$
(we choose any one of the two complex values of $\sqrt{a_2}$).
Then we have $\wt\varphi_3=\wt x_2^2+\dots$, where dots mean a
quadratic form in $\wt x_0$, $\wt x_1$.
The latter can be
killed by a suitable change of variables in the target space (more precisely,
by adding a certain linear combination of the coordinates $u_0$, $u_1$, $u_2$
to the last coordinate $u_3$).
Thus we reduced $\Phi$ to the form
$[\wt x_0:\wt x_1:\wt x_2]\mapsto [\wt x_0^2:\wt x_0\wt x_1:\wt x_1^2:\wt x_2^2]$,
i.e., to the form $\Phi_3$.
\end{proof}

Finally, we need to distinguish between $\Phi_2$ and $\Phi_3$, i.e.,
to prove that these two maps are not equivalent.
To this end, it suffices to compute the dual planarizations of
$\Phi_2$ and $\Phi_3$ and observe that equivalent planarizations must
have equivalent dual planarizations.
The planarization $\Phi_2$ is co-trivial: the image of every line
is contained in a plane passing through $[0:0:1:0]$.
On the other hand, a straightforward computation shows that the dual
planarization of $\Phi_3$ is equivalent to $\Phi_3$, in particular, is
not trivial.
This concludes the proof of Theorem \ref{t:qclass}.

\subsection{Real classification}
\label{ss:R}
In this subsection, we assume that the ground field is $\R$.
The proof of Theorem \ref{t:qclassR} splits into the following two lemmas.

\begin{lem}
\label{l:R-nonsing}
  Suppose that $S$ is nonsingular.
  Then $\Phi$ is equivalent to one and only one of the following maps:
  \begin{align*}
  \Phi_{1a}:&[x_0:x_1:x_2]\mapsto [x_0^2:x_0x_1:x_0x_2:x_1x_2],\\
  \Phi_{1b}:&[x_0:x_1:x_2]\mapsto [x_0^2:x_0x_1:x_0x_2:x_1^2+x_2^2].
  \end{align*}
\end{lem}

\begin{proof}
  Consider the linear web of conics $\Lc_\Phi$ associated with the
  complexification of $\Phi$.
  By the proof of Lemma \ref{l:C-nonsing}, the web $\Lc_{\Phi}$ has two
  different complex base points $a\ne b$ and consists of all conics
  passing through $a$ and $b$.
  There are two possibilities: $a$ and $b$ can be real or complex conjugate.
  Suppose first that $a$ and $b$ are real.
  Then, as in the proof of Lemma \ref{l:C-nonsing}, we show that
  $\Phi$ is equivalent to $\Phi_{1a}$.
  Suppose now that $a$ and $b$ are complex conjugate.
  Performing a suitable change of homogeneous coordinates $[x_0:x_1:x_2]$ in $\P^2$
  over real numbers, we may assume that $a=[0,1,i]$ and $b=[0,1,-i]$.
  Then the web of all conics passing through $a$ and $b$ is spanned by
  the degenerate conics $x_0^2=0$, $x_0x_1=0$, $x_0x_2=0$ and $x_1^2+x_2^2=0$.
  Note that, in the affine chart $x_0=1$ with affine coordinates $x_1$, $x_2$,
  the web $\Lc_\Phi$ is exactly the web of all circles.
  With this choice of generating conics, $\Phi$ coincides with $\Phi_{1b}$.
\end{proof}

\begin{lem}
\label{l:R-sing}
  Suppose that $S$ is singular but irreducible.
  Then $\Phi$ is equivalent to one and only one of the following maps:
  \begin{align*}
  \Phi_{2}:&[x_0:x_1:x_2]\mapsto [x_0^2:x_0x_1:x_1^2:x_2x_0]\\
  \Phi_{3}:&[x_0:x_1:x_2]\mapsto [x_0^2:x_0x_1:x_1^2:x_2^2].
  \end{align*}
\end{lem}

\begin{proof}
  The proof of Lemma \ref{l:C-sing} applies verbatim over reals,
  except that it is not always possible to take $\sqrt{a_2}$.
  If $a_2<0$, then we set instead $c_2=1/\sqrt{-a_2}$ and reduce
  $\Phi$ to the form
  $[\wt x_0:\wt x_1:\wt x_2]\mapsto [\wt x_0^2:\wt x_0\wt x_1:\wt x_1^2:-\wt x_2^2]$.
  However, changing sign of the last coordinate in $\P^3$ gets it
  back to the form $\Phi_3$.
  Since $\Phi_2$ and $\Phi_3$ are not equivalent over complex numbers,
  neither are they over reals.
\end{proof}

\section{Normal forms of planarizations}
\label{s:norf}

In this section, we give a complete list of local normal forms of planarizations.
We assume that the ground field is $\R$.

\begin{thm}
  \label{t:nf}
  Suppose that $U\subset\P^2$ is an open subset and $\Phi:U\to\P^3$ is a planarization.
  Then, for every open subset $V\subset U$ there exists a $($possibly smaller$)$
  open subset $W\subset V$ such that $\Phi:W\to\P^3$ is projectively equivalent to at least one of the following normal forms:
  \begin{itemize}
\item[$(T)$:] $[x:y:z]\mapsto [f(x,y,z):g(x,y,z):h(x,y,z):0]$
\item[$(CT)$:] $[x:y:z]\mapsto [x:y:z:f(x,y,z)]$
\item[$(Q_1)$:] $[x:y:z]\mapsto [xy:xz:yz:x^2+y^2+z^2]$
\item[$(Q_2)$:] $[x:y:z]\mapsto [xy:xz:yz:x^2 - y^2 + z^2]$
\item[$(Q_3)$:] $[x:y:z]\mapsto [x^2+y^2:y^2+z^2:xz:yz]$
\item[$(Q_4)$:] $[x:y:z]\mapsto [x^2-y^2:xy:yz:z^2]$
\item[$(Q_5)$:] $[x:y:z]\mapsto [xz-yz:x^2:y^2:z^2]$
\item[$(Q_6)$:] $[x:y:z]\mapsto [x^2:xz-y^2:yz:z^2]$
\item[$(Q_7)$:] $[x:y:z]\mapsto [y^2-z^2:xy:xz:yz]$
\item[$(Q_8)$:] $[x:y:z]\mapsto [xy:xz:y^2:z^2]$
\item[$(Q_9)$:] $[x:y:z]\mapsto [xy:xz-y^2:yz:z^2]$
\item[$(Q_{10})$:] $[x:y:z]\mapsto [x^2:xy:y^2:z^2]$
\item[$(C_1)$:] $[x:y:z]\mapsto [z(x^2+y^2):y(x^2+z^2):x(y^2+z^2):xyz]$
\item[$(C_2)$:] $[x:y:z]\mapsto [z(x^2-y^2):y(x^2+z^2):x(y^2-z^2):xyz]$
\item[$(C_3)$:] $[x:y:z]\mapsto [x^2z:z(x^2+y^2):x(x^2+y^2-z^2):y(x^2+y^2+z^2)]$
\item[$(C_4)$:] $[x:y:z]\mapsto [x^2y:x(x^2-y^2):z(x^2+y^2):yz^2]$
\item[$(C_5)$:] $[x:y:z]\mapsto [x^2(x+y):y^2(x+y):z^2(x-y):xyz]$
\item[$(C_6)$:] $[x:y:z]\mapsto [x^3:xy^2:2xyz-y^3:z(xz-y^2)]$.
\end{itemize}
Here $f$, $g$ and $h$ are sufficiently smooth degree 1 homogeneous functions of $x$, $y$, $z$.
In normal form $(T)$, the mapping $[x:y:z]\mapsto [f:g:h]$ represents an arbitrary
sufficiently smooth embedding of $W$ into $\P^2$.
\end{thm}

In Theorem \ref{t:nf}, the form $(T)$ represents all trivial planarizations,
and the form $(CT)$ represents all co-trivial planarizations.
These items correspond to infinitely many projective equivalence classes.
However, note that there are finitely many classes of nontrivial non-co-trivial
planarizations.

\begin{proof}
  By the Main Theorem, every point $a\in U$ has a neighborhood $V\subset U$
  such that the planarization $\Phi:V\to\P^3$ is trivial, co-trivial, quadratic
  or dual quadratic.
  Suppose first that $\Phi:V\to\P^3$ is trivial.
  This means by definition that there is a plane $P\subset\P^3$ such that
  $\Phi(V)\subset P$.
  By a projective coordinate change, we may assume that the plane $P$
  is given in homogeneous coordinates $[u_0:u_1:u_2:u_3]$ by $u_3=0$.
  Then the map $\Phi:V\to\P^3$ has form $(T)$.

  Suppose now that the planarization $\Phi:V\to\P^3$ is co-trivial but not trivial.
  Then, by definition, there is a point $b\in\P^3$ such that, for every line $L\subset\P^2$,
  there is a plane $P_L$ containing the set $\Phi(L\cap V)\cup\{b\}$.
  We may assume that $b=[0:0:0:1]$.
  Since $\Phi:V\to\P^3$ is not trivial, there is a nonempty open subset $W\subset V$
  such that $\Phi(W)\not\ni b$.
  Let $\P^2(b)$ be the projective plane formed by all lines in $\P^3$ passing through $b$,
  and let $\pi:\P^3\sm\{b\}\to\P^2(b)$ be the canonical projection mapping a point
  $c\ne b$ to the line $bc$.
  Note that the map $\Psi=\pi\circ\Phi:W\to\P^2(b)$ has the following property.
  For every line $L\subset\P^2$, the set $\Psi(W\cap L)$ is a subset of a line.
  By the M\"obius--von Staudt theorem, a map with this property must be
  a restriction of a projective transformation or a mapping from $W$ to a line,
  possibly after replacing $W$ with a smaller open set.
  In the second case, $\Phi:W\to\P^3$ is trivial, and therefore is equivalent to
  the form $(T)$.
  In the first case, the map $\Psi$ is given by $[x:y:z]\mapsto [x:y:z]$
  provided that we choose a suitable system of projective coordinates in $\P^2(b)$.
  Then the map $\Phi:W\to\P^3$ is given by
  $[x:y:z]\mapsto [x:y:z:f(x,y,z)]$ for some (sufficiently smooth) function $f$.

  Suppose now that the planarization $\Phi:V\to\P^3$ is quadratic but
  neither trivial nor co-trivial.
  Note that the image $\Phi(V)$ lies in a surface $S$ of degree 2, 3 or 4.
  If $S$ has degree 2, then $\Phi$ is equivalent to one of the maps $\Phi_{1a}$, $\Phi_{1b}$,
  $\Phi_2$, $\Phi_3$ from Theorem \ref{t:qclassR}.
  Since, by our assumption, $\Phi$ is not co-trivial, it must be equivalent to $\Phi_3$;
  the latter is redenoted by $(Q_{10})$ in the statement of the theorem.
  Suppose now that $S$ has degree 3 or 4.
  In this case, we refer to the results of \cite{CSS}.
  By \cite{CSS}, every quadratic rational map $\Phi$ such that $\Phi(\P^2\sm B_\Phi)$ is dense
  in a surface of degree 3 or 4 is equivalent to one of the maps $(Q_1)$--$(Q_9)$.

  Finally, suppose that the planarization $\Phi:V\to\P^3$ is dual quadratic but
  not trivial, not co-trivial, and not quadratic.
  Then its dual planarization is equivalent to one of the maps $(Q_1)$--$(Q_{13})$.
  A straightforward computation shows that the dual planarizations to
  $(Q_1)$--$(Q_6)$ are equivalent, respectively, to $(C_1)$--$(C_6)$.
  The planarizations $(Q_7)$--$(Q_9)$, characterized by the property that the
  corresponding surfaces in $\P^3$ are cubic, turn out to be equivalent to
  their dual planarizations.
  In particular, the dual planarizations of $(Q_7)$--$(Q_9)$ are quadratic.
\end{proof}

The equations of the surfaces parameterized by dual-quadratic planarizations
$(C_1)$--$(C_6)$ are
\begin{itemize}
\item[$(C_1)$:] $4t^3 - t(u^2 + v^2 + w^2) + uvw = 0$
\item[$(C_2)$:] $4t^3 + t(u^2 - v^2 + w^2) + uvw = 0$
\item[$(C_3)$:] $4vu^2 + u(t^2 - 4v^2 + w^2) - vw^2 = 0$
\item[$(C_4)$:] $4tu^2 - uw^2 + tv^2 = 0$
\item[$(C_5)$:] $u(vw - t^2) + vt^2 = 0$
\item[$(C_6)$:] $u(4tv - w^2) + v^3 = 0$,
\end{itemize}
where $[u:v:w:t]$ are homogeneous coordinates in $\P^3$. 
These equations are obtained by eliminating the three variables $x$, $y$, $z$ from
the four equations
$$
[u:v:w:t]=\Phi[x:y:z].
$$
We used \textit{Mathematica} to perform the computations. 
We provide figures of these surfaces below, see Figures 1--3.

The surfaces that are parameterized by maps $(Q_1)$--$(Q_9)$ have been studied in \cite{CSS}.
In particular, pictures of all these surfaces can be found there.

\begin{figure}[H]
 \includegraphics[height=4cm]{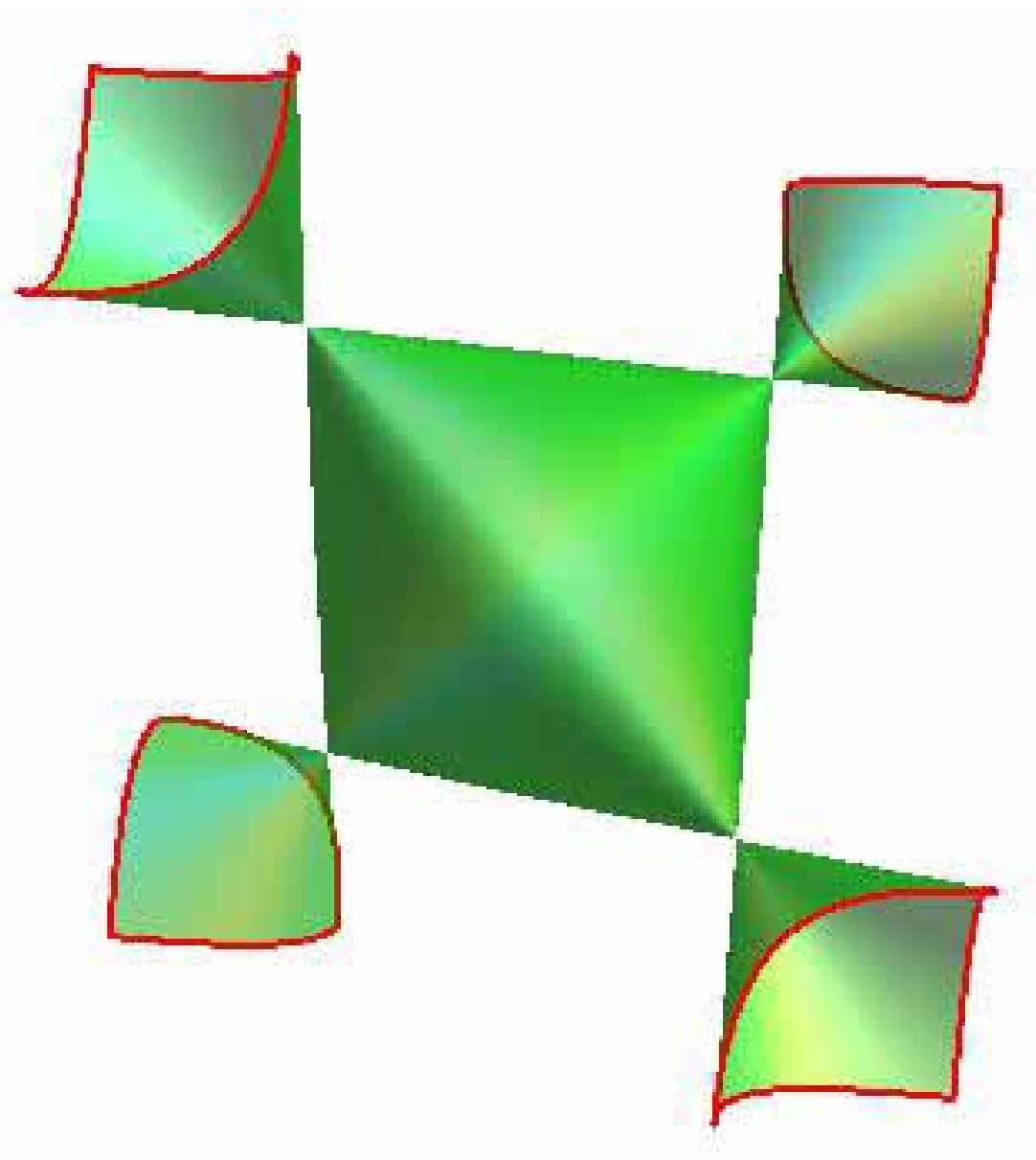}\hspace{1cm}
 \includegraphics[height=4cm]{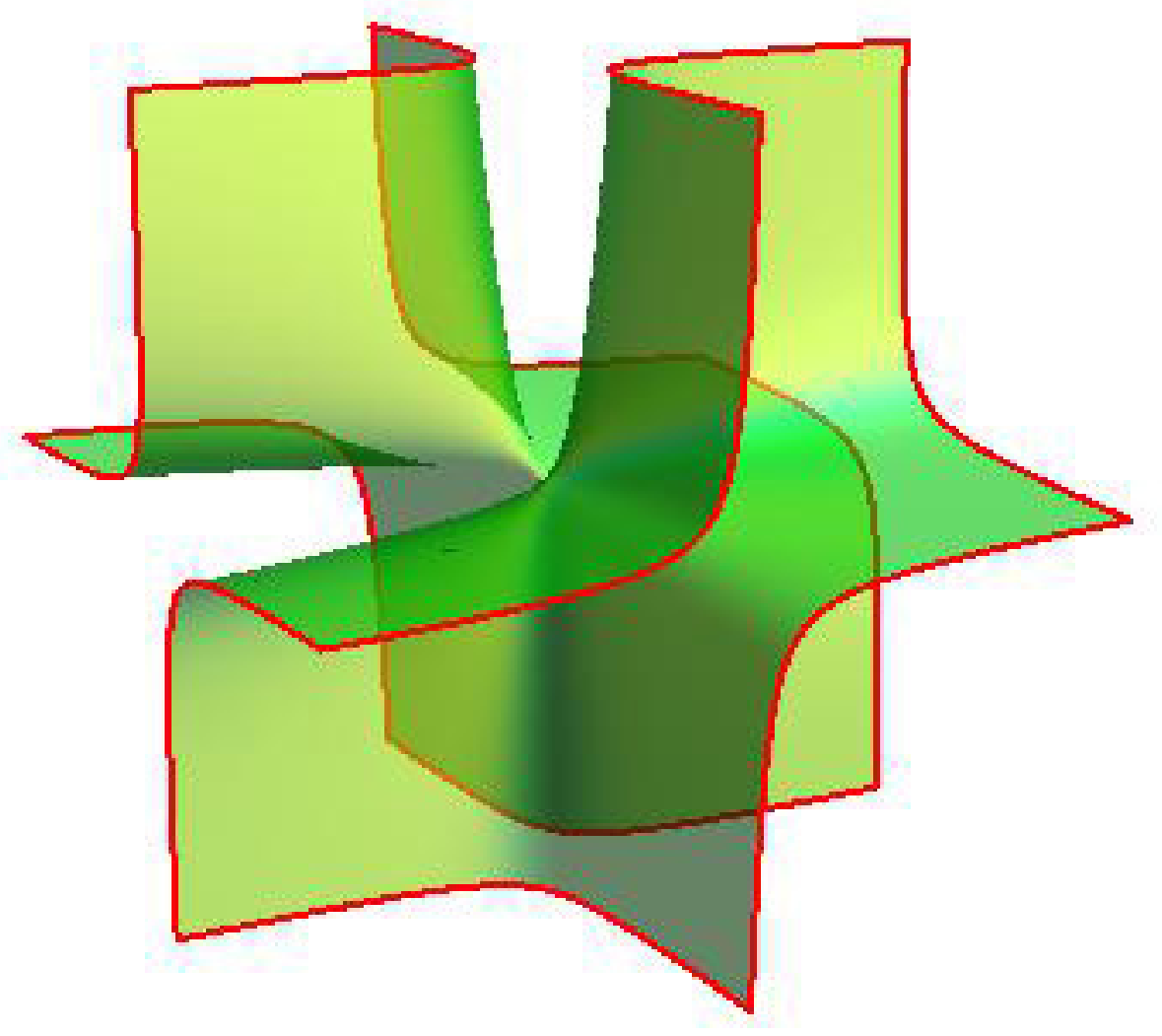}
\caption{The surfaces parameterized by $(C_1)$ (left) and by 
$(C_2)$ (right) in the affine chart $t=1$.}
\end{figure}

\begin{figure}[H]
 \includegraphics[height=4cm]{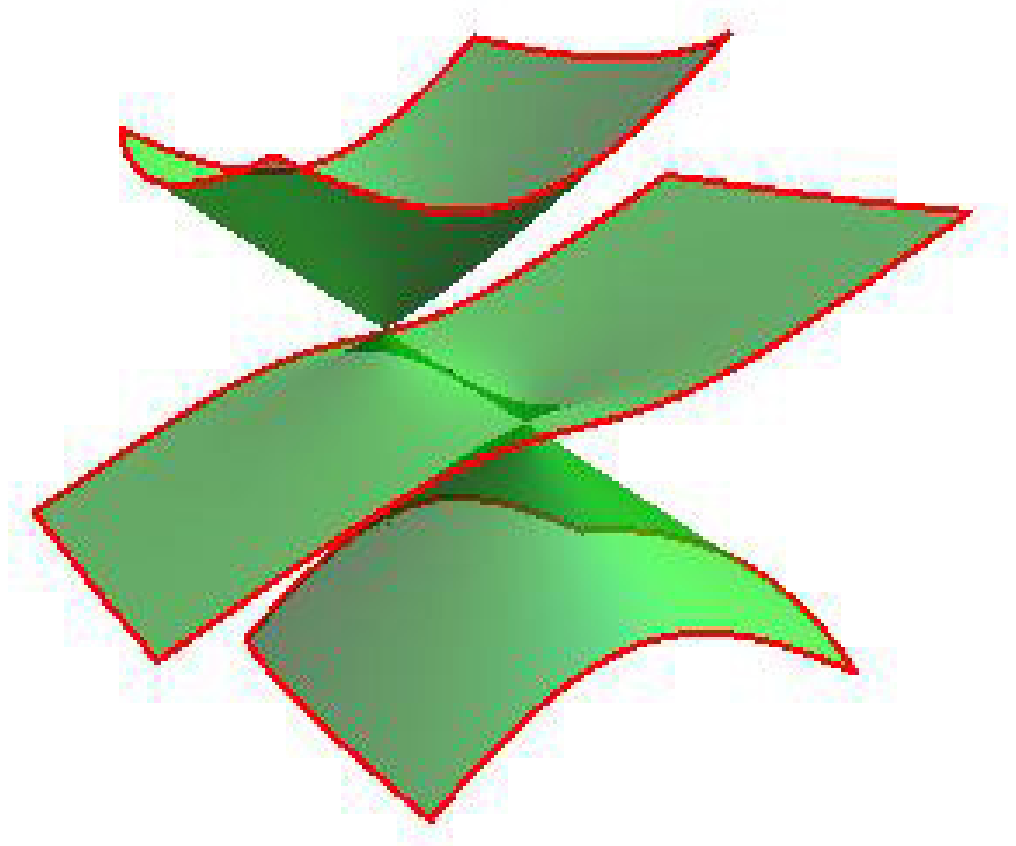}\hspace{1cm}
 \includegraphics[height=4cm]{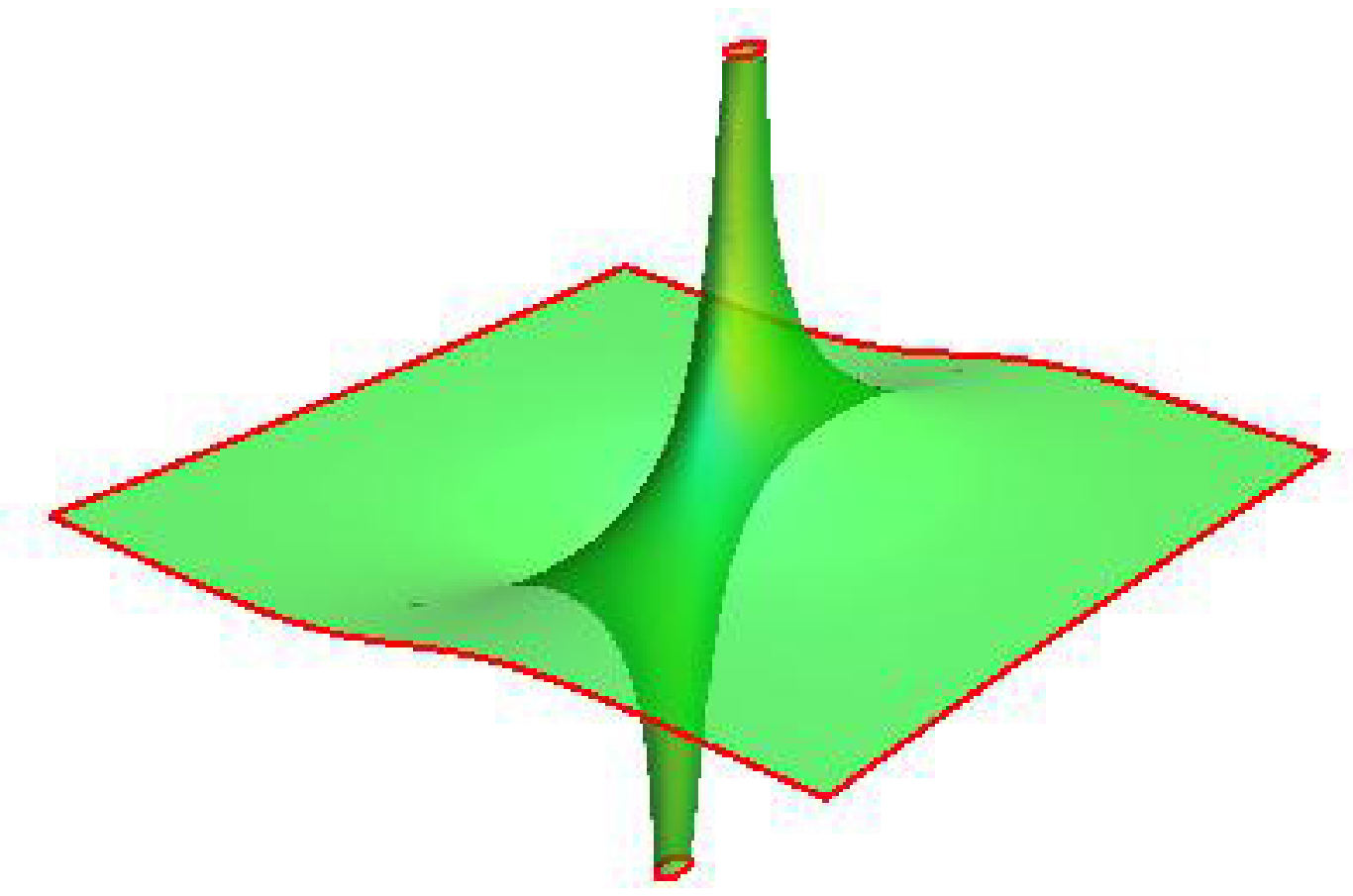}
\caption{The surfaces parameterized by $(C_3)$ in the affine chart $t=1$ (left)
and by $(C_4)$ in the affine chart $w=1$ (right).}
\end{figure}

\begin{figure}[H]
 \includegraphics[height=4cm]{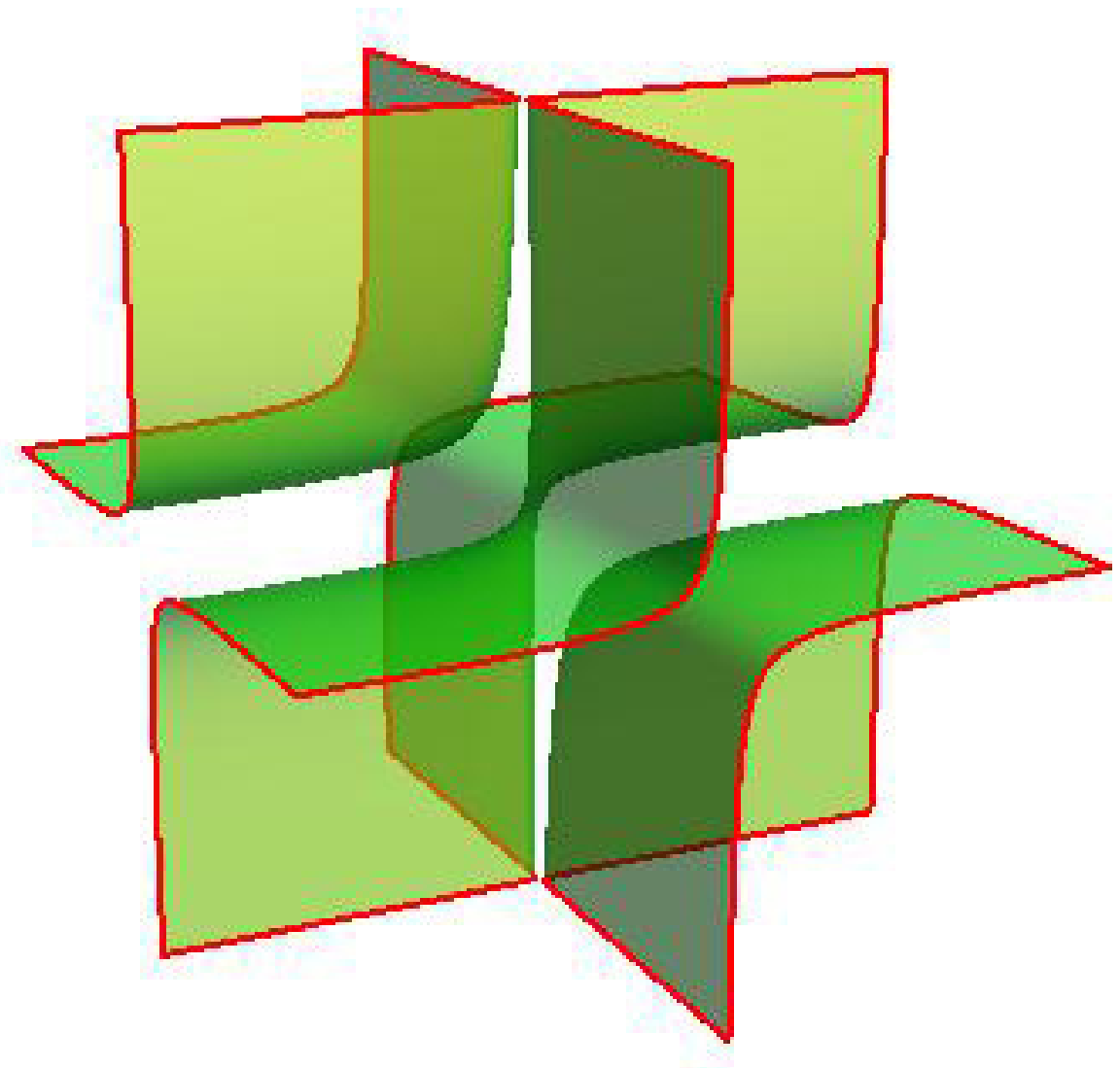}\hspace{1cm}
 \includegraphics[height=4cm]{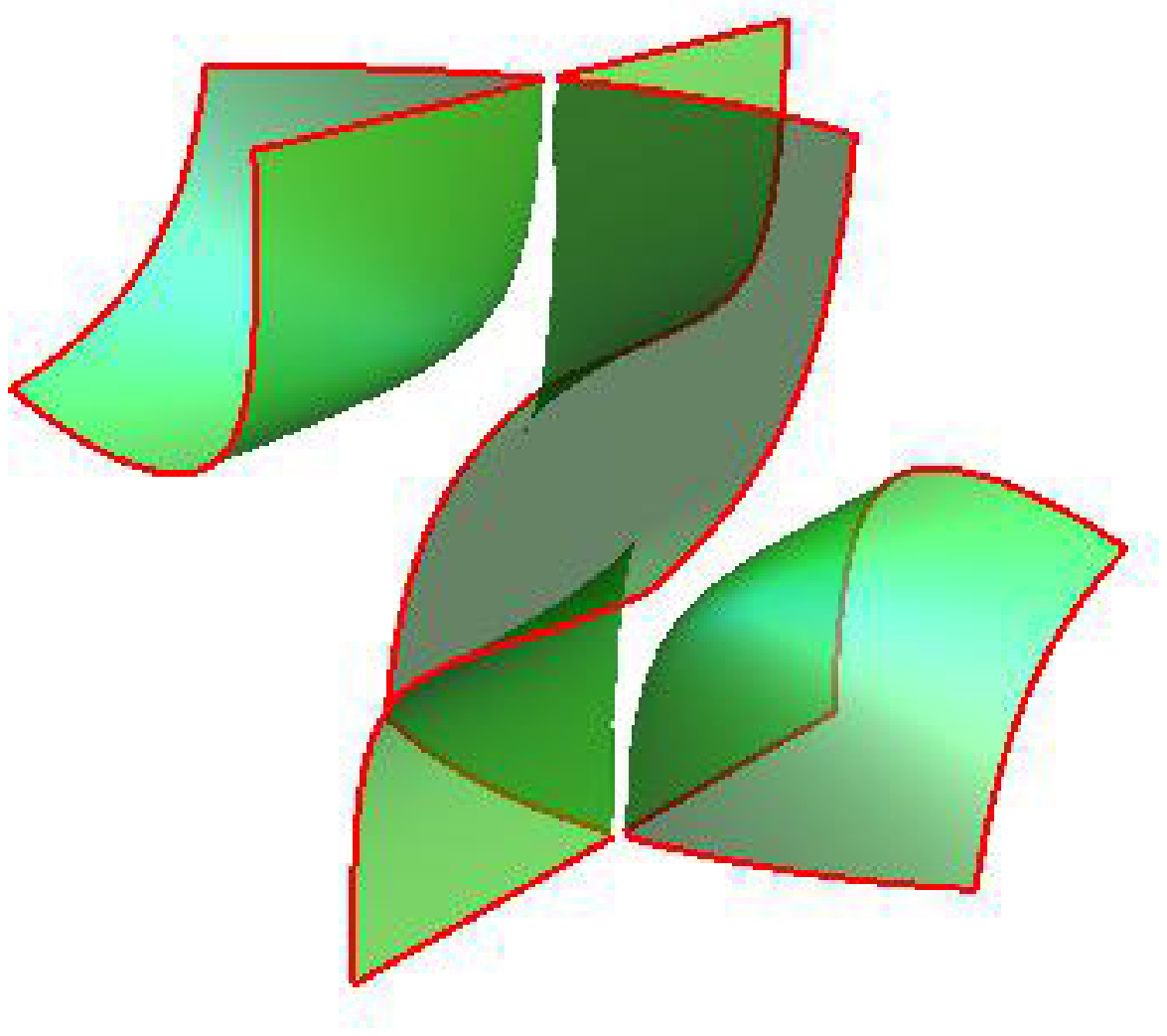}
\caption{The surface parameterized by $(C_5)$ in the affine chart $t=1$ (left)
and by $(C_6)$ in the affine chart $w=1$ (right).}
\end{figure}

\end{document}